\documentclass[12pt]{amsart}
\usepackage{amsmath, amssymb, amsthm, latexsym}
\usepackage{mathrsfs}
\usepackage{amssymb}
\usepackage{latexsym}
\usepackage[center]{caption}
\usepackage[OT2,T1]{fontenc}
\DeclareSymbolFont{cyrletters}{OT2}{wncyr}{m}{n}
\DeclareMathSymbol{\Sha}{\mathalpha}{cyrletters}{"58}
\usepackage{tikz}
\newcounter{braid}
\newcounter{strands}
\pgfkeyssetvalue{/tikz/braid height}{1cm}
\pgfkeyssetvalue{/tikz/braid width}{1cm}
\pgfkeyssetvalue{/tikz/braid start}{(0,0)}
\pgfkeyssetvalue{/tikz/braid colour}{black}
\pgfkeys{/tikz/strands/.code={\setcounter{strands}{#1}}}

\makeatletter
\def\cross{%
  \@ifnextchar^{\message{Got sup}\cross@sup}{\cross@sub}}

\def\cross@sup^#1_#2{\render@cross{#2}{#1}}

\def\cross@sub_#1{\@ifnextchar^{\cross@@sub{#1}}{\render@cross{#1}{1}}}

\def\cross@@sub#1^#2{\render@cross{#1}{#2}}

\def\render@cross#1#2{
  \def\strand{#1}
  \def\crossing{#2}
  \pgfmathsetmacro{\cross@y}{-\value{braid}*\braid@h}
  \pgfmathtruncatemacro{\nextstrand}{#1+1}
  \foreach \thread in {1,...,\value{strands}}
  {
    \pgfmathsetmacro{\strand@x}{\thread * \braid@w}
    \ifnum\thread=\strand
    \pgfmathsetmacro{\over@x}{\strand * \braid@w + .5*(1 - \crossing) * \braid@w}
    \pgfmathsetmacro{\under@x}{\strand * \braid@w + .5*(1 + \crossing) * \braid@w}
    \draw[braid] \pgfkeysvalueof{/tikz/braid start} +(\under@x pt,\cross@y pt) to[out=-90,in=90] +(\over@x pt,\cross@y pt -\braid@h);
    \draw[braid] \pgfkeysvalueof{/tikz/braid start} +(\over@x pt,\cross@y pt) to[out=-90,in=90] +(\under@x pt,\cross@y pt -\braid@h);
    \else
    \ifnum\thread=\nextstrand
    \else
     \draw[braid] \pgfkeysvalueof{/tikz/braid start} ++(\strand@x pt,\cross@y pt) -- ++(0,-\braid@h);
    \fi
   \fi
  }
  \stepcounter{braid}
}

\tikzset{braid/.style={double=\pgfkeysvalueof{/tikz/braid colour},double distance=1pt,line width=2pt,white}}

\newcommand{\braid}[2][]{%
  \begingroup
  \pgfkeys{/tikz/strands=2}
  \tikzset{#1}
  \pgfkeysgetvalue{/tikz/braid width}{\braid@w}
  \pgfkeysgetvalue{/tikz/braid height}{\braid@h}
  \setcounter{braid}{0}
  \let\sigma=\cross
  #2
  \endgroup
}
\makeatother

\input xypic

\swapnumbers
\newtheorem{theorem}{Theorem}

\newtheorem{lemma}[theorem]{Lemma}


\def\Z{\mathbb{Z}}

\def\C{\mathbb{C}}

\def\C{\mathbb{C}}

\def\N{\mathbb{N}}

\def\F{\mathbb{F}}

\def\qed{\hfill$\square$\medskip}

\def\Zpk{\mathbb{Z}/p^{k}}
\def\Zpk1{\mathbb{Z}/p^{k-1}}

\newcommand{\rref}[1]{(\ref{#1})}

\newcommand{\beg}[2]{\begin{equation}\label{#1}#2\end{equation}}
\def\r{\rightarrow}

\def\F{\mathbb{F}}

\def\sl2{\widetilde{SL_{2}(\Z)}}

\title[Equivariant filtration on $THH$ of polynomial 
algebras]{On the equivariant motivic filtration of the topological Hochschild 
homology of polynomial algebras}
\author{Po Hu, Igor Kriz and Petr Somberg}
\thanks{The authors acknowledge support by grants GA\,CR P201/12/G028 and GA\,CR 19-28628X.
Kriz also acknowledges the support from the Simons Foundation, most recently through Awards 403297 and 958219.}

\begin{document}
\maketitle

\vspace{5mm}

{\em Dedicated to the fond memory of Jan Nekov\'{a}\v{r} with thanks for the many cheerful and 
enlightening conversations we had in Prague in recent years.}

\vspace{5mm}

\begin{abstract}
We identify the equivariant structure of the filtered pieces
of the motivic filtration defined by
Bhatt, Morrow and Scholze on the topological Hochschild cohomology spectrum
of polynomial algebras over $\F_p$.
\end{abstract}

\

\noindent

\section{Introduction}\label{s1}
This note stems from the authors' discussions with Peter Scholze in the period of 2015-2017.
In connection with questions on the structures defined in the paper \cite{bsm},
Bhatt, Morrow and Scholze
\cite{bsm2} defined an $S^1$-equivariant decreasing filtration $F^iTHH(R)$ on the topological Hochschild
homology of a smooth $\F_p$-algebra $R$, which they called the {\em motivic filtration}.  In this note,
by $THH(R)$ we mean the genuine $RO(S^1)$-graded equivariant spectrum. (The existence
of such a canonical structure was proved
by B\"{o}kstedt, Hsiang and Madsen \cite{bhm}.)
The purpose of this note is to compute the full equivariant structure of the filtered pieces $F^iTHH(R)$,
which strengthens the results of \cite{bsm2} on these spectra.

The coefficients (homotopy groups) of the $\Z/p^{r-1}$-fixed
points of $THH(R)$ can be expressed in terms of 
$$TR^r(\F_p)=\Z/p^r[\sigma_r].$$
The answer \cite{hesselholt, hesselholt1} is
$$TR^r(\F_p)\otimes_{\Z/p^r} W_r \Omega R$$ 
where $W_r\Omega$ denotes the de Rham-Witt complex of length $r$. The motivic filtration
is a decreasing multiplicative filtration in which $\sigma_r$ has filtration degree $1$, and on $ W_r\Omega$, the filtration
coincides with the de Rham filtration. It is shown in \cite{bsm}
that the associated graded spectrum of the motivic
filtration, (and therefore its fixed point spectra $F^iTR^r(R)$), are equivariant (resp. non-equivariant)
Eilenberg-MacLane spectra. 

\vspace{3mm}
However, computation of the equivariant structure of $F^i TR^r(R)$ remains an interesting open problem.
While for cyclotomic spectra, the equivariant structure can be deduced from non-equivariant information (\cite{nik}),
$F^i TR^r(R)$ are, in fact, {\em not} cyclotomic spectra (see \cite{bsm2}).
The purpose of
this note is to calculate one example, namely the equivariant structure on $F^iTHH(R)$ in the case where $R$ is a
polynomial algebra over $\F_p$ on finitely many generators.

\vspace{3mm}

The method \cite{bsm2} of definining the motivic filtration is by proving {\em semiperfect descent}.
For $R$ smooth over $\F_p$, if we 
denote by $R_{perf}$ the perfectization of $R$, then the canonical homomorphism
$R\r R_{perf}$, equalizing the two canonical homomorphisms $R_{perf}
\rightrightarrows R_{perf}\otimes_R R_{perf}$, 
induces an equivalence of $S^1$-equivariant
spectra
\beg{en1}{THH(R)\r |THH(R_{perf}\otimes_R\dots\otimes_R R_{perf})|_\bullet
} 
where 
$|?|_\bullet$ means cosimplicial realization. Now the coefficients
of the $\Z/m$-fixed points
of the constituent spaces on right hand side of 
\rref{en1} (where $m$ is a natural number)
are concentrated in even dimensions, since the rings $R_{perf}\otimes_R\dots\otimes_R R_{perf}$
are quasiregular semiperfect $\F_p$-algebras (meaning that the Frobenius is onto, the first Andre-Quillen homology is
a flat module and the higher Andre-Quillen homology is $0$). 
The desired filtration is obtained by applying cosimplicial realization to the equivariant Postnikov
filtration of those spectra.

\vspace{3mm}
To prove that \rref{en1} is an equivalence on $\Z/m$-fixed points, Bhatt, Morrow and 
Scholze \cite{bsm2} do not use a direct calculation. They used the fact that since $THH$ is a cyclotomic 
spectrum, it suffices to prove that \rref{en1} induces an isomorphism on homotopy co-fixed points (i.e. the corresponding
Borel homology). Taking the left derived functor of the equivariant Postnikov filtration of $THH(R)$ in the
category of simplicial $\F_p$-algebras $R$ (i.e. the right Kan extension of the functor
given by the equivariant Postnikov filtration on $THH(?)$ 
from the source category 
of simplicial smooth $\F_p$-algebras to the category
of $\Z/m$-equivariant $S$-modules), and smashing it with $E\Z/m_+$, produces associated graded pieces which
are smash-products of pieces of (the spectral realization of) the Quillen cotangent complex $L$ with trivial $\Z/m$-action,
smashed with $E\Z/m_+$. For these pieces, the analogue of the descent equivalence \rref{en1} follows
from Bhargav Bhatt's theorem on faithfully flat descent for $L$ (\cite{bhatt,scholze}).

\vspace{3mm}
The equivalence \rref{en1} gives rise to a spectral sequence obtained by the cosimplicial filtration
on the right hand side.
A direct calculation of this spectral sequence is an interesting exercise worth going through for 
its own sake. For a general $R$ say, a smooth $\F_p$-algebra, it is still somewhat out of
range at this point. We do it for $R=\F_p[x_1,\dots,x_n]$, which, as we will note, can be used to
obtain a proof of semi-perfect descent which is different from the argument of \cite{bsm2}. More immediately
from the point of interest of this note, for a polynomial algebra, the spectral sequence \rref{en1} collapses to
$E_2$, which can be used to calculate the equivariant structure on $F^iTHH(R)$. This is the main purpose of
our note.

\vspace{3mm}
As a warm-up,
let us discuss the case of $n=1$ first. 
To simplify notation, let us abbreviate the notation $THH(R)$ to $T(R)$. Using the
method of Hesselholt and Madsen \cite{hmpoly}, one sees that, $S^1$-equivariantly,
\beg{e1}{T(\F_p[x])=S^0\vee\bigvee_{n\in \N} S^1_{n+}\wedge T(\F_p)
}
where $S^1_n$ is $S^1$ with the $S^1$-action by $n$-th power, $\N=\{1,2,\dots\}$. 
Note that we have a cofibration sequence
(hence, stably, equivalently a fibration sequence)
$$S^1_{n+}\r S^0\r \widetilde{S^1_n}$$
where $\widetilde{?}$ denotes unreduced suspension. Note that $\widetilde{S^1_n}$ is also the $1$-point
compactification of the irreducible representation $\C[\xi_n]$ on which an element $z\in S^1$ acts by multiplication
by $z^n$. Let us denote the Bhatt-Morrow-Scholze 
motivic filtration by $F^i$, and let $?_{\geq i}$ denote the equivariant Postnikov
filtration.

\begin{theorem}\label{t1}
Let $p$ be a prime, and let $m\in \N$. For $R=\F_p[x]$,
the spectral sequence associated with applying $\Z/m$-fixed points to the target of \rref{en1} collapses
to $E_2$, and converges to its abutment 
$T(R)^{\Z/m}$. Additionally, $F^iT(\F_p[x])$ is (via the canonical map) equivalent to the 
wedge of $T(\F_p)_{\geq i}$ with the wedge over $n\in \N$ of homotopy fibers of the maps
$$T(\F_p)_{\geq i}\r (\widetilde{S^1_n}\wedge T(\F_p))_{\geq i}.$$
\end{theorem}

We shall prove Theorem \ref{t1} in detail, noting the intricacies of the spectral sequence involved. 
However, precisely by the same method, one can also prove a generalization:

\begin{theorem}\label{t2}
Let $p$ be a prime, and let $m\in \N$. 
For $R=\F_p[x_1,\dots,x_n]$, the spectral sequence obtained by applying $\Z/m$-fixed points
to \rref{en1} collapses to $E_2$ and converges to its abutment
$T(R)^{\Z/m}$. Additionally, $F^iT(R)$ is the wedge, over subsets $A\subseteq \{1,\dots,n\}$, and maps
$n_?:A\r\N$
of the $|A|$-fold homotopy fibers of the cubes $(S^{\bigoplus_{s\in A}\C[\xi_{n_s}]}\wedge T(\F_p))_{\geq i}$.
\end{theorem}

\qed

All the objects in sight are well known to satisfy \'{e}tale descent. This can be used to give another proof 
of the flat descent equivalence \rref{en1}, and also a method for calculating the motivic filtration explicitly for smooth
$\F_p$-algebras. In effect, if an $\F_p$-algebra $S$ is \'{e}tale over an $\F_p$-algebra $R$, then
we have
$$S_{perf}=S\otimes_R R_{perf}.$$
Thus, we may form a bi-cosimplicial spectrum

\beg{eddouble}{THH(S_{perf}
\otimes_S\dots \otimes_S S_{perf}\otimes_R R_{perf}\otimes_R \dots\otimes_R R_{perf})}
(the $(0,0)$'th term is $THH(S_{perf}\otimes_R R_{perf}))=THH(S_{perf}
\otimes_S S_{perf})$. Assuming \rref{en1} is an equivariant
equivalence for $R$, realizing in the $R$-coordinate of \rref{eddouble} 
first, we get the right hand side of \rref{en1} for $S$. 
Realizing the $S$-coordinate of \rref{eddouble}
first, (and accepting that for a perfect ring, its $THH$ can be computed by
replacing it by the 
level-wise $THH$ of its cosimplicially equivalent semiperfect resolution), we get 
$$THH(S\otimes_R R_{perf}\otimes_R\dots \otimes_R R_{perf}),$$
which realizes to $THH(S)$ (see \cite{scholze}).

\vspace{3mm}
\noindent
{\bf Acknowledgement:} We are thankful to Peter Scholze for discussions on this subject.

\vspace{5mm}
\section{The spectral sequence}

By usual $THH$ considerations, the interesting mathematics happens on $\Z/p^{r-1}$-fixed points,
$r\in\N$. We will restrict
attention to this case.
In fact, one puts
$$TR^r(R)=T(R)^{\Z/p^{r-1}}.$$
Let us recall that Hesselholt and Madsen \cite{hmwitt}, Theorem 5.5, computed
$$TR^r(\F_p)=\Z/p^r[\sigma].$$
More precisely, the $\Z/p^{r-1}$-equivariant Postnikov tower of $T(\F_p)$ consists of even suspensions of the equivariant
Eilenberg-MacLane spectrum of the Mackey (in fact, Green) functor $\mathscr{W}_r(\F_p)$ whose
value on the orbit with isotropy $\Z/p^{i-1}$ is $W_i(\F_p)$, and Mackey functor restrictions are given
by the Frobenius, while transfers are given by Verschiebung, with isotropy groups acting trivially. An analogous
Mackey functor can, in fact, be defined with $\F_p$ replaced by any commutative ring $A$. In the present case, 
the Frobenius is just reduction modulo the appropriate power of $p$, while Verschiebung is multiplication by $p$.
(The Frobenius, which is the Mackey functor ``restriction", is not to be confused with the cyclotomic spectrum
restriction, which is the map from $\Z/p^{r-1}$-fixed points to the $\Z/p^{r-2}$-fixed points of the geometric
fixed points, which is equivalent to the $\Z/p^{r-2}$-fixed points of the original cyclotomic spectrum. In our present
setting, this map is a homomorphism of rings which sends $\sigma$ to $p \sigma$.)

\vspace{3mm}
The coefficients of $\Z/{p^{r-1}}$-fixed points
of \rref{e1} can be calculated either directly, or one can also use the calculation of Illusie \cite{illusie}, 
\textsection 1.2
of the deRham-Witt complex of affine spaces, using Hesselholt's formula for a regular $\F_p$-algebra $R$
\beg{ehess}{TR^r(R)_{n}=\bigoplus_{j+2k=n,j,k\in \N_0}W_r\Omega^jR
}
where $W_r\Omega^*R$ is the (truncated) deRham-Witt complex.
In fact, 
$$F^iTR^r(R)$$ 
is the sum of those summands of \rref{ehess} for which $j+k\geq i$.

\vspace{3mm}
In the present case, the deRham-Witt complex of $\F_p[x]$ is a sum, over $n\in \N_0$
(corresponding to the power of $x^n$) of the exterior algebra on one generator (in dimension $1$)
over $\Z/p^r$ tensored
with the submodule ${}_{p^{i+1}}W_r(\F_p)$ of $W_r(\F_p)$ consisting 
of elements which are annihilated by $p^{i+1}$, where $p^i$ is the highest
power of $p$ which divides $n$:
$$
\bigoplus_{n\in\N_0,\;p^i||n}\Lambda_{\Z/p^r}[u]\otimes ({}_{p^{i+1}}W_r(\F_p))
$$
(we put $i=r$ for $n=0$).
Thus, this module is $\Z/{p^{\min(i+1,r)}}[\sigma]$.

\vspace{3mm}
We have
$$\F_p[x]_{perf}=\F_p[x^{1/p^\infty}]:=\lim_\r \F_p[x^{1/p^n}].$$
Taking colimits, denoting by $\Z[1/p]^+_0$ the set of non-negative elements in $\Z[1/p]$,
one has, (somewhat non-canonically), 
\beg{e2}{T(\F_p[x^{1/p^\infty}])=\bigvee_{s\in \Z[1/p]^+_0} T(\F_p).
}
One can write
$$\underbrace{\F_p[x^{1/p^\infty}]\otimes_{\F_p[x]}\dots\otimes_{\F_p[x]} \F_p[x^{1/p^\infty}]}_{\text{$(n+1)$
times}}=
\F_p[x^{1/p^\infty}]\otimes ( \F_p[y^{1/p^\infty}]/(y))^{\otimes n}.
$$
Denoting the copies of $x$ on the left hand side by $x_i$, $i=1,\dots, n+1$, the copies $y_i$, $i=1,\dots, n$
of $y$ are given by
$$y_i^{1/p^k}=x_i^{1/p^k}-x_{i+1}^{1/p^k}.$$

\vspace{3mm}

\begin{lemma}\label{l1}
\beg{e3}{T(\F_p[y^{1/p^\infty}]/(y))=\bigvee_{s\in \Z[1/p]^+_0} S^{V_{\lfloor s\rfloor}}\wedge T(\F_p)
}
where 
$$V_n=\C[\xi_1]\oplus\dots\oplus \C[\xi_n].$$
(Thus, 
$S^{V_n}=\widetilde{S^1_1}\wedge\dots \wedge \widetilde{S^1_n}
$.)
\end{lemma}

\begin{proof}
Use the results of \cite{hesselholt} (Theorem 11) or \cite{hmwitt} (Proposition 9.1)
and pass to the colimit, keeping in mind that the colimit of $\Z/p^{k}_+$ can
be identified with $S^1_+$ due to the fact that we are smashing with the $p$-complete spectrum $T(\F_p)$.
\end{proof}

\vspace{3mm}
(Note: The formulas \rref{e2}, \rref{e3} are meant $S^1$-equivariantly.)

\vspace{3mm}
On fixed points, this can be described purely algebraically. Recall the result of Hesselholt 
\cite{hesselholt} (Theorem 11) or Hesselholt and Madsen \cite{hmwitt} (Proposition 9.1)
stating that for a complex $\Z/p^{r-1}$-representation $V$,
\beg{ehess10}{\begin{array}{c}
(S^V\wedge T(\F_p))^{\Z/p^{r-1}}_{2a}=\Z/p^{i} \;
\\
\text{for the unique $i$ such that
$|V^{\Z/p^{r-i}}|\leq a< |V^{\Z/p^{r-i-1}}|$}.\end{array}
}
The odd-degree groups are $0$. In \rref{ehess10}, for bookkeeping reasons, we set
$|V^{\Z/p^{-1}}|=\infty$, $|V^{\Z/p^{r}}|=-\infty$. Using \rref{ehess10} and Lemma \ref{l1}, one finds that 
\beg{e3alg}{\begin{array}{l}
(T(\F_p[y^{1/p^\infty}]/(y)))^{\Z/p^{r-1}}_{*}=\\[2ex]
\displaystyle
\bigoplus_{\ell\in \N_0}\sigma^\ell\cdot \Z/(p^r)[z_\ell^{1/p^\infty}]/(z_\ell^{\ell+1},
pz_\ell^{(\ell+1)/p},\dots, p^{r-1}z_\ell^{(\ell+1)/p^{r-1}})
\end{array}
}
where $\sigma$ has degree $2$
and the ring structure is explained by $pz_{\ell+1}=z_\ell$, $z_0=y$. (One uses the techniques of 
\cite{hesselholt,hmwitt}; note that the summand $s$ in the statement of Lemma \ref{l1} corresponds
to the term $y^{s/p^r}$; also note the reference \cite{posledni}, which explains how to name 
classes in $TR^{n+1}_{*-V}$ via differential forms.) Keep in mind, however, that
in the ring, $z_\ell$ only occurs after being multiplied by $\sigma^\ell$. Note that this is not in the
form \rref{ehess}, but then again, $ \F_p[y^{1/p^\infty}]/(y)$ is not a regular ring. Factoring
the ideal $(\sigma)$ out of \rref{e3alg}, we obtain the truncated Witt vectors:
\beg{e3witt}{W_r(\F_p[y^{1/p^\infty}]/(y))=
\Z/(p^r)[y^{1/p^\infty}]/(y,py^{1/p},\dots, p^{r-1}y^{1/p^{r-1}}).
}

\vspace{3mm}
Now back to the spectral sequence.
To verify that \rref{en1} induces
an isomorphism on $\Z/p^{r-1}$-fixed points for $R=\F_p[x]$, 
we must prove that the canonical map from \rref{en1} to the cobar construction on \rref{e2}
by the co-algebra \rref{e3} is an equivalence on $\Z/p^{r-1}$-fixed points. We obtain a spectral
sequence from applying $\Z/p^{r-1}$-fixed points to the right hand side of \rref{en1} and taking coefficients
$$
E_1^{*,*}=T
(\F_p[x^{1/p^\infty}]\otimes ( \F_p[y^{1/p^\infty}]/(y))^{\otimes *})^{\Z/p^{r-1}}_*
$$
converging to $T(?)^{\Z/p^{r-1}}_*$ applied to the cosimplicial realization.
We shall study this spectral sequence. (Note on grading: regardless of the position of the indices, because
the present situation mixes homological and cohomological grading, we grade everything homologically. A cohomological
degree can be converted to homological by taking its negative.)

Now the $E_1$-term can be simplified by writing \rref{e3} as
$$(\bigvee_{s\in\Z[1/p]\cap [0,1)} T(\F_p))\wedge_{T(\F_p)} (\bigvee_{n\in \N_0} S^{V_n}\wedge T(\F_p)).$$
Thereby, the $E_1$-term of the spectral sequence has a tensor factor of the dual Koszul complex of
$$\bigvee_{s\in\Z[1/p]\cap [0,1)} T(\F_p)^{\Z/p^{r-1}}_*,$$
which can be factored out since its homology is concentrated in degree $0$.
This leads to a simpler $E_1$-term which is the cosimplicial $\Z/{p^r}$-module 
(with cosimplicial coordinate $k\in \N_0$)
\beg{e4}{((\bigvee_{m\in \N_0} S^0)\wedge (\bigvee_{n\in\N_0} S^{V_n})^{\wedge k}\wedge 
T(\F_p))^{\Z/p^{r-1}}_*
}
Thus, we need to prove that the homology of the chain complex \rref{e4} coincides
with (an associated graded object of) the coefficients of the $\Z/p^{r-1}$-fixed points of \rref{e1},
which we recalled above.

\vspace{3mm}
Let us first note that for $r=1$, this works. \rref{e4} is just the cobar construction of $\Z/p[t]$
over a polynomial coalgebra on one generator of (homologically graded) bidegree $(-1,2)$ (with trivial coaction),
which has (homologically graded) cohomology 
\beg{e5}{T(\F_p[x])_*=\Z/p[t,\sigma]\oplus \Sigma \Z/p[t,\sigma]}
where $\Sigma$ is the
suspension (by the calculation of $Ext^*_{A[z]}(A,A)$ for a commutative ring $A$), which is the correct answer.

\vspace{3mm}
To make the calculation on $\Z/p^{r-1}$-fixed points, we 
note that the $E_1$-term \rref{e4} is graded by
$$d=m+n_1+\dots+n_k.$$
We will work in each degree $d$ separately. Assume $d=p^j d^\prime$ where $d^\prime$ is not divisible by $p$.
We will be working in each of the ``$\Z/p^i$-regions" (according to \rref{ehess10}) of the degree $d$ part of the
$E_1$-term \rref{e4} separately, showing that for the most part, they will cancel out one at a time.

\vspace{3mm}
Filtering \rref{e4} by terms with $m$ less or equal to a given number, we obtain a (purely algebraic) spectral sequence
whose $E_2$-term has the differential formed by the $\geq 1$'th co-faces of \rref{e4}. Now the method 
of computing this $E_2$-term 
is based on the computation of
\beg{eext}{Ext_{\Z/p^i[v]}(\Z/p^i,\Z/p^i),
}
with some exceptions caused by the varying $i$ in \rref{ehess10},
Concretely, the term \rref{ehess10} with $i=r$, in a given dimension, is computed literally by the formula
\rref{eext}, by using the non-equivariant result, and comparison with Hesselholt's computation
of homotopy fixed points \cite{hesselholt}, in which the coefficients embed. The 
answer is the exterior algebra over $\Z/p^r$ on one element $(v)$:
$$
\Lambda_{\Z/p^r}[v],
$$ 
which is represented by the generator
of 
$k=1$, $n=1$ in \rref{e4}. The $0$-coface component of the $d_1$ differential of \rref{e4} maps the 
$\Z/p^r$ in the same dimension for $m=d$, $k=0$ by $p^j$, giving the right answer in dimensions 
$\geq 1$.

\vspace{3mm}
To understand what happens for $i<r$, it is useful to filter $\Z/p^i[v]$ by an increasing filtration in which
$v,v^p,\dots, v^{p^{r-1}}$ are given filtration degree $1$. The associated graded algebra then will be
\beg{essgr}{\Z/p^i[v]/(v^p)\otimes \Z/p^i[v^p]/(v^{p^2})\otimes\dots \otimes\Z/p^i[v^{p^{r-1}}].
}
The $Ext(\Z/p^i,\Z/p^i)$ over \rref{essgr} is
\beg{essgr1}{\Lambda_{\Z/p^i}((v_0),\dots,(v_{r-1}))\otimes \Z/p^i[t_1,\dots,t_{r-1}]
}
where $v_b$ is dual to $v^{p^b}$
and the $t_b$ is the transpotence element occurring in the well known calculation
of the $Ext$ of truncated
polynomial algebras, which in our case is in cohomological dimension $2$ and $v$-degree $p^b$. The (purely algebraic)
spectral sequence calculating the $Ext(\Z/p^i,\Z/p^i)$ over $\Z/p^i[v]$ has the standard differentials
\beg{essgd}{d^1((v_b))=t_b, \;b=1,\dots, r-1
}
(obtained from the cobar construction model - see for example \cite{ravenel}, Lemma 3.2.4).
Now in our topological Hochschild homology situation the total topological dimension (which is always even)
determines $b$, and which of the differentials \rref{essgd} are disrupted. In fact, using \rref{ehess10} and
comparing the dimensions of the $\Z/p^i$-fixed points of the two sides of \rref{essgd} where $v_{m}$ is 
replaced by $V_{p^m}$ (the multiplication corresponds to direct sum)
shows that the only place where the differentials are disrupted for the reason of the
two sides being in different $j$-ranges of \rref{ehess10} is for $a=0$. This means that we are
dealing with the \rref{e3witt} part of \rref{e3alg}. 

\vspace{3mm}
The additional differentials in the spectral sequence obtained by taking $\Z/(p^{r-1})$-fixed points of 
\rref{e4}, (i.e. the part not explained by the coproduct on the middle term), are explained by the comodule
structure of the leftmost term. Going back to the notation \rref{e2}, \rref{e3witt}, this structure
is given by
\beg{exy}{x\mapsto x\otimes 1 +1\otimes y.
}
In the above analysis, we omitted the power of $x$ from the notation, but the power of $x$ present is
equal to the number $m$ in \rref{eext}. Depending on its $p$-valuation, we obtain additional
$d^1$-differentials, which wipe out all the remaining terms of the spectral sequence, except the answer. 

\vspace{3mm}
Concretely, 
let $d$ be divisible by $p^j$ but not $p^{j+1}$ for some $0\leq j<r$. (If $d$ is divisible by $p^r$, then only the 
$i<j$ part of the below differential pattern will occur.) Then the contribution
of \rref{exy} to the $d_1$ differential
pattern in topological dimension $0$ of the spectral sequence \rref{e4}, which is not explained
by the $Ext(\Z/p^i,\Z/p^i)$ over the algebra \rref{essgr}, consists of the degree $\leq d$ part of
\beg{eddiff1}{t_{i+1}^k\cdot (V_{p^i})\mapsto t_{i+1}^{k+1}, \; k\geq 0}
tensored with 
\beg{eddiff2}{\Lambda[(V_{p^\ell}),\ell>i]\otimes \Z/p^?[t_{\ell+1},\ell>i]
}
for $0\leq i<j$, and the degree $\leq d$ part of
\beg{eddiff3}{t_{j+1}^k\mapsto (V_{p^{j}})\cdot t_{j+1}^{k},\; k\geq 0}
tensored with 
\beg{eddiff4}{\Lambda[(V_{p^\ell}),\ell>j]\otimes \Z/p^?[t_{\ell+1},\ell>j].
}
This cancels all the remaining elements,
with the exception of  \rref{eddiff3} with $k=0$, where the source group is $\Z/p^r$, while the
target group is $\Z/p^{r-j-1}$, thus leaving a kernel of $\Z/p^{j+1}$, as claimed.

\vspace{5mm}
\section{Proof of Theorem \ref{t2}}\label{spt2}

First note that \rref{e1} generalizes to
\beg{e1m}{T(\F_p[x_1,\dots,x_\ell])=\bigvee_{A\subseteq\{1,\dots,\ell\}}
\bigvee_{n\in \N^A}S^{\bigoplus_{s\in A}\xi_{n(s)}}\wedge T(\F_p).
}
To prove Theorem \ref{t2}, we will proceed by induction on $\ell$. To this end, it is actually advantageous to
smash \rref{e1m} with $S^V$ for an arbitrary finite-dimensional complex representation $V$. Using the equivalence
between a diagonal totalization and a $k$-fold successive totalization of a $\ell$-fold cosimplicial abelian group,
the $E_1$-term of the spectral sequence analogous to \rref{e4} actually becomes an $\ell$-fold chain complex
of the form
\beg{e4m}{
(S^V\wedge (\bigvee_{m\in\N_0}S^0)^\ell
\wedge (\bigvee_{n\in \N_0}S^{V_{n}})^{k_1}\wedge\dots
\wedge (\bigvee_{n\in \N_0}S^{V_{n}})^{k_\ell}\wedge T(\F_p))_*^{\Z/p^{r-1}}
}
We will show that the $E_2$-term is actually isomorphic to the coefficients of
\beg{e4m1}{(S^V\wedge T(\F_p[x_1,\dots,x_\ell]))_*^{\Z/p^{r-1}}}
which, in turn, splits into summands over various choices of $A$, and the choices of mappings $n:A\r \N$. Let us
first state what the answer is:

\begin{lemma}\label{ltt21}
The summand of \rref{e4m1} corresponding to a map $n:A\r \N$ is isomorphic to a sum of $2j$-suspensions of
the exterior algebra over ${\Z/p^{i^\prime}}$ on $|A|$ generators where 
\beg{erange2}{|V^{\Z/p^{i+1-i^\prime}}|
\leq j< |V^{\Z/p^{i-i^\prime}}|}
where $i$ is the highest non-negative integer such that $p^i$ divides all the numbers $n(s)$, $s\in A$, 
and $r-1$. Here we apply the
same boundary conventions as before. In particular, the summand is $0$ for $j<|V^{\Z/p^{i}}|$.
\end{lemma}

\begin{proof}
The proof proceeds along similar lines as in the $\ell=1$ case, but we skip it, since it is a part of 
Hesselholt's results \cite{hesselholt}.
\end{proof}

Now as already mentioned, \rref{e4m} can be considered as an $\ell$-fold chain complex, which splits into
summands corresponding to maps $n:A\r\N$, which we can
study one at a time. We will proceed by induction on $\ell$.  Since the complexes for proper subsets are isomorphic to summands of \rref{e4m} for lower $\ell$, without loss of generality, we can assume $A=\{1,
\dots,\ell\}$. Now we will proceed by calculating, for each $n$, first the homology of the complex 
obtained by totalizing the last $\ell-1$ differentials. This will lead to another purely algebraic
spectral sequence, but again, we will see that it collapses to $E_2$. In fact, the $E_1$-term has already
been calculated by the induction hypothesis, where $V$ is replaced by 
$$V\oplus V_{n_1}\oplus\dots\oplus V_{n_{k_1}}$$
where $n_1,\dots, n_{k_1}$ are the $n$'s occurring in the wedge summand in the $k_1$-smash power in 
\rref{e4m}.

On the other hand, by Lemma
\ref{ltt21},
the answer is (the appropriate summand of) the exterior algebra over $\Z_p$ on $\ell-1 $ generators,
tensored with the $E_1$-term of the spectral sequence \rref{e4m} for $\ell=1$, with $r$ replaced by the minimum of
$r$ and $i+1$, where $i$ is the highest non-negative integer such that $p^i$ divides all of the numbers
$n(2),\dots,n(\ell)$. 

\vspace{3mm}
Thus, what remains to do is to figure out the $d_1$ of the spectral sequence \rref{e4m} with $\ell=1$, or,
equivalently, the spectral sequence \rref{e4} where the argument is smashed with $S^V$, or, explicitly,
\beg{e441}{(S^V\wedge(\bigvee_{m\in \N_0} S^0)\wedge (\bigvee_{n\in\N_0} S^{V_n})^{\wedge k}\wedge 
T(\F_p))^{\Z/p^{r-1}}_*
}
This is a variation of our analysis of the $d_1$ of the spectral 
sequence \rref{e4} in the last section. Essentially, the discussion of
the cases is exactly the same, except that if 
$$|V^{\Z/p^{r-i}}|\leq m< |V^{\Z/p^{r-i-1}}|$$
(with the same conventions as under \rref{ehess10}), $r$ is replaced by $i+1$ and $i$ is replaced by 
another index.

Department of Mathematics, Wayne State University, 1150 Faculty/

\noindent
Administration Bldg., Detroit, MI 48202

\vspace{3mm}

Department of Mathematics, University of Michigan, 530 Church Street, Ann Arbor, MI 48109-1043

\vspace{3mm}

Mathematical Institute of Charles University, Faculty of Mathematics and Physics, Sokolovsk\'{a} 83, Praha 8, 
Karl\'{\i}n, 18000 Czech Republic

\vspace{3mm}

Corespondence to be sent to: kriz.igor@gmail.com


\vspace{10mm}
\begin{thebibliography}{99}

\bibitem{bhatt} B.Bhatt: Completions and derived de Rham cohomology, preprint, arXiv: 1207.6193

\bibitem{bsm} B.Bhatt, M.Morrow, P.Scholze: Integral $p$-adic Hodge theory, preprint, arXiv: 1602.03148

\bibitem{bsm2} B.Bhatt. M.Morrow, P.Scholze: Topological Hochschild Homology and
Integral $p$-adic Hodge theory, 
arXiv:1802.03261

\bibitem{bhm} M.B\"{o}kstedt, W.C.Hsiang, I.Madsen:
The cyclotomic trace and algebraic K-theory of spaces, {\em Invent. Math.} 111 (1993), no. 3, 465-539

\bibitem{hesselholt} L.Hesselholt: K-theory of truncated polynomial algebras,
{\em Handbook of K-theory}, Vol. 1, 2, 71-110, Springer, Berlin, 2005

\bibitem{hesselholt1} L.Hesselholt: On the $p$-typical curves in Quillen's K-theory, {\em Acta Math.} 177
(1996), 1-53

\bibitem{hmwitt} L.Hesselholt, I.Madsen: On the K-theory of finite algebras over Witt vectors of perfect fields, 
{\em Topology} 36 (1997), no. 1, 29-101

\bibitem{hmpoly} L.Hesselholt, I.Madsen: Cyclic polytopes and the K-theory of truncated polynomial algebras, 
{\em Invent. Math.} 130 (1997), no. 1, 73-97

\bibitem{posledni} L. Hesselholt, I. Madsen: On the K-theory of nilpotent endomorphisms,
{\em Homotopy Methods in Algebraic Topology (Boulder, CO, 1999)}, Contemporary
Mathematics, vol. 271, American Mathematical Society, Providence, RI,
2001, pp. 127-140

\bibitem{hks} P.Hu, I.Kriz, P.Somberg: Derived representation theory and stable homotopy categorification of $sl_k$,
{\em Adv. in Math.}, Volume 341, 7 January 2019, 367--439,
https://doi.org/10.1016/j.aim.2018.10.044

\bibitem{illusie} L.Illusie: Complexe de de\thinspace Rham-Witt et cohomologie cristalline, 
{\em Ann. Sci. \'{E}cole Norm. Sup.} (4) 12 (1979), no. 4, 501-661

\bibitem{nik} T.Nikolaus, P. Scholze: On topological cyclic homology, {\em Acta Math.} 221 (2018), no. 2, 203-409,
Correction:  {\em Acta Math.} 222 (2019), no. 1, 215-218

\bibitem{ravenel} D.C.Ravenel:
{\em Complex cobordism and stable homotopy groups of spheres}.
Pure and Applied Mathematics, 121. Academic Press, Inc., Orlando, FL, 1986

\bibitem{scholze} P.Scholze: private communications to the authors, 2015-2018

\end{thebibliography}
\end{document}